\documentclass[12pt]{amsart}

\usepackage[T1]{fontenc}
\usepackage[utf8]{inputenc}
\usepackage[a4paper, hmargin={2cm}]{geometry}

\usepackage{lmodern}

\usepackage{bm}

\usepackage{amssymb, amsfonts}
\usepackage{paralist} %
\usepackage{thm-restate}

\usepackage[english]{babel} %
\usepackage{csquotes} %

\usepackage[backend=bibtex, firstinits=true, style=alphabetic, doi=false, url=false, eprint=true, sorting=nyt, maxnames=3, maxbibnames=99]{biblatex}
\usepackage{tikz, tikz-cd}
\usepackage{xcolor}
\usepackage[colorlinks=true, citecolor=blue, urlcolor=violet, breaklinks=true]{hyperref}
\usepackage[capitalize]{cleveref}

\newbibmacro{string+doiurl}[1]{%
	\iffieldundef{doi}{%
		\iffieldundef{url}{%
			\iffieldundef{eprint}{%
				#1%
			}{\href{https://arxiv.org/abs/\thefield{eprint}}{#1}}
		}{\href{\thefield{url}}{#1}}
	}{\href{https://dx.doi.org/\thefield{doi}}{#1}}
}

\renewcommand*{\bibfont}{\footnotesize}
\DeclareFieldFormat{eprint:arxiv}{arXiv\addcolon\space \nolinkurl{#1}}
\DeclareFieldFormat[book]{title}{\usebibmacro{string+doiurl}{\mkbibemph{#1}}}
\DeclareFieldFormat[article,incollection, inproceedings,online, thesis,misc]{title}{\usebibmacro{string+doiurl}{#1}}

\newtheorem{theorem}{Theorem}[section]

\newtheorem{corollary}[theorem]{Corollary}
\newtheorem{proposition}[theorem]{Proposition}

\theoremstyle{definition}

\newtheorem{question}[theorem]{Question}
\newtheorem{conjecture}[theorem]{Conjecture}
\newtheorem{example}[theorem]{Example}

\newcommand{\ab}{\mathbf{a}}
\newcommand{\bb}{\mathbf{b}}
\newcommand{\eb}{\mathbf{e}}
\newcommand{\xb}{\bm{x}}
\newcommand{\mm}{\mathfrak{m}}

\newcommand{\NN}{\mathbb{N}}
\newcommand{\ZZ}{\mathbb{Z}}

\newcommand{\kk}{\Bbbk}

\newcommand{\tC}{\widetilde{C}}
\newcommand{\Fb}[1][\bullet]{\mathbb{F}_{#1}}
\DeclareMathOperator{\lin}{lin}
\newcommand{\Fblin}[1][\bullet]{\lin(\mathbb{F}_{#1})}

\newcommand{\set}[1]{\{#1\}}
\newcommand{\sset}[1]{{#1}} %
\newcommand{\with}{\,\colon\,}

\DeclareMathOperator{\Tor}{Tor}
\DeclareMathOperator{\Span}{span}

\newcommand{\restr}[2]{{#2}|_{#1}}

\newcommand{\too}{\longrightarrow}
\newcommand{\oot}{\longleftarrow}

\newcommand{\<}{\langle}
\renewcommand{\>}{\rangle}

\newcommand{\del}{\mathrm{d}}
\newcommand{\plin}{d^{\mathrm{lin}}}

\newcommand{\coho}[2]{\widetilde{H}^{#1}(#2)} 

\newcommand{\ho}[2]{\widetilde{H}_{#1}(#2)} 

\begin{document}

\title[Linear maps in minimal free resolutions]{Linear maps in minimal free resolutions of Stanley-Reisner rings}
\author{Lukas Katth\"an}
\address{Goethe-Universit\"at, FB 12 -- Institut f\"ur Mathematik, Postfach 11 19 32, D--60054 Frankfurt am Main, Germany}

\email{katthaen@math.uni-frankfurt.de}


\begin{abstract}
	In this short note we give an elementary description of the linear part of the minimal free resolution of a Stanley-Reisner ring of a simplicial complex $\Delta$.
	Indeed, the differentials in the linear part are simply a compilation of restriction maps in the simplicial cohomology of induced subcomplexes of $\Delta$.
	
	Along the way, we also show that if a monomial ideal has at least one generator of degree $2$, then the linear strand of its minimal free resolution can be written using only $\pm 1$ coefficients.
\end{abstract}

\maketitle

\section{Introduction}
Let $\kk$ be a field and $S = \kk[x_1, \dotsc, x_n]$ be a polynomial ring over it.
Consider a finitely generated graded $S$-module $M$, and its minimal free resolution $\Fb$.
The \emph{linear part} \cite{EFS} $\Fblin$ of $\Fb$ has the same modules as $\Fb$, and its differential $\plin$ is obtained from the differential $\del$ of $\Fb$ by deleting all non-linear entries in the matrices representing $\del$ in some basis of $\Fb$.

The main result of this short note is an explicit description of $\Fblin$ in the case where $M = \kk[\Delta]$ is the Stanley-Reisner ring of a simplicial complex $\Delta$.
It is well-known that $\Fb$ is multigraded and generated as $S$-module in squarefree multidegrees.
For simplicity we identify squarefree multidegrees with subsets of $[n] := \set{1,\dotsc,n}$.
We are going to use Hochster's formula, which states that
\[ \Tor_i^S(\kk[\Delta],\kk)_U \cong \coho{\#U - i - 1}{\Delta_U; \kk}, \]
where $U \subseteq [n]$ is a squarefree multidgree and $\Delta_U := \set{F \in \Delta\with F \subseteq U}$ is the \emph{restriction} of $\Delta$.
To simplify the notation, we set $U \setminus u := U \setminus \set{u}$ and $\coho{j}{\Delta_U} := \coho{j}{\Delta_U; \kk}$ for $u \in U \subseteq[n]$ and $j \in \NN$.
By Hochster's formula, $\Fblin[i]$ is isomorphic to the direct sum of modules of the form $\coho{\#U - i - 1}{\Delta_U} \otimes_\kk S(-U)$.
The differential $\plin$ turns out to be simply a compilation of all the restriction maps $\coho{i}{\Delta_U} \to \coho{i}{\Delta_{U\setminus\sset{u}}}, \omega \mapsto \restr{U\setminus\sset{u}}{\omega}$, induced by the inclusions $\Delta_{U\setminus\sset{u}} \subset \Delta_U$.
In the following theorem, we use the notation $\alpha(u,U) = \#\set{k \in U, k < u}$, where $u \in [n]$. 

\begin{restatable}{theorem}{thmmain}\label{thm:main}
	Let $\kk[\Delta]$ be the  Stanley-Reisner ring of a simplicial complex $\Delta$ and let $\Fb$ denote its minimal free resolution. 
	The linear part $\Fblin$ of $\Fb$ is isomorphic to the complex with modules 
	\[ \Fblin[i] = \bigoplus_{U \subseteq [n]} \coho{\#U - i - 1}{\Delta_U} \otimes_\kk S(-U),\]
	and the components of the differential are given by
	\[
	\begin{aligned}
	\coho{j}{\Delta_U} &\otimes_\kk S(-U) &&\too& \coho{j}{\Delta_{U\setminus\sset{u}}} &\otimes_\kk S(-{U\setminus\sset{u}}) \\
	\omega &\otimes s &&\longmapsto& (-1)^{\alpha(u,U)} \restr{U\setminus\sset{u}}{\omega} &\otimes x_u s
	\end{aligned}
	\]
\end{restatable}
This extends the result of Reiner and Welker \cite[Theorem 3.2]{RW}, which describes the maps in the linear strand of $\Fb$.
An alternative description of $\Fblin$ in terms of the Alexander dual of $\Delta$ was given by Yanagawa \cite[Theorem 4.1]{Y00}.

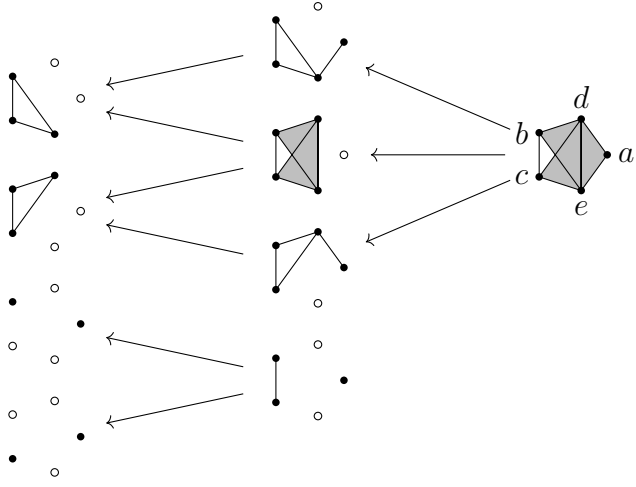
\begin{figure}[t]
	\begin{tikzpicture}[scale=0.5]
	\newcommand{\setcoords}{
		\coordinate (a) at (0:1);
		\coordinate (b) at (144:1);
		\coordinate (c) at (216:1);
		\coordinate (d) at (72:1);
		\coordinate (e) at (-72:1);
	}
	\newcommand{\drawblack}[1]{\foreach \p in {#1} \fill \p circle (0.1);}
	\newcommand{\drawwhite}[1]{\foreach \p in {#1} \draw \p circle (0.1);}
	\begin{scope}
	\setcoords
	
	\draw[fill=gray!50]
	(a)--(d)--(e)--cycle
	(c)--(d)--(e)--cycle
	(b)--(d)--(e)--cycle;
	
	\draw (b) -- (c);
	
	\drawblack{(a),(b),(c),(d),(e)}
	
	\path (a) node[anchor=west] {$a$};
	\path (b) node[anchor=east] {$b$};
	\path (c) node[anchor=east] {$c$};
	\path (d) node[anchor=south] {$d$};
	\path (e) node[anchor=north] {$e$};
	\end{scope}
	
	\begin{scope}[xshift=-7cm]
		\begin{scope}[yshift=3cm]
			\setcoords
			
			\draw (c) -- (b) -- (e) -- cycle;
			\draw (a) -- (e);
			
			\drawblack{(a),(b),(c),(e)}
			\drawwhite{(d)}
		\end{scope}
		
		\begin{scope}[yshift=0cm]
			\setcoords
			
			\draw[fill=gray!50] 
			(c)--(d)--(e)--cycle
			(b)--(d)--(e)--cycle; 
			
			\draw (b) -- (c);
			
			\drawblack{(b),(c),(d),(e)}
			\drawwhite{(a)}
		\end{scope}
		
		\begin{scope}[yshift=-3cm]
			\setcoords
			
			\draw (b) -- (c) -- (d) -- cycle;
			\draw (a) -- (d);
			
			\drawblack{(a),(b),(c),(d)}
			\drawwhite{(e)}
		\end{scope}
		
		\begin{scope}[yshift=-6cm]
			\setcoords
			
			\draw (b) -- (c);
			
			\drawblack{(a),(b),(c)}
			\drawwhite{(d),(e)}
		\end{scope}
	\end{scope}
	
	\begin{scope}[xshift=-14cm,yshift=-1.5cm]
		\begin{scope}[yshift=3cm]
			\setcoords
			
			\draw (b) -- (c) -- (e) -- cycle;
			
			\drawblack{(b),(c),(e)}
			\drawwhite{(a),(d)}
		\end{scope}
		
		\begin{scope}[yshift=0cm]
			\setcoords
			
			\draw (b) -- (c) -- (d)-- cycle;
			
			\drawblack{(b),(c),(d)}
			\drawwhite{(a),(e)}
		\end{scope}
		
		\begin{scope}[yshift=-3cm]
			\setcoords
			
			\drawblack{(a),(b)}
			\drawwhite{(c),(d),(e)}
		\end{scope}
		
		\begin{scope}[yshift=-6cm]
			\setcoords
			
			\drawblack{(a),(c)}
			\drawwhite{(b),(d),(e)}
		\end{scope}
	\end{scope}
	
	\begin{scope}[every node/.style={circle,minimum size = 1.7cm}]
		\node at(0,0) (abcde) {};
		\node at (-7,3) (abce) {};
		\node at (-7,0) (bcde) {};
		\node at (-7,-3) (abcd) {};
		\node at (-7,-6) (abc) {};
		\node at (-14,1.5) (bce) {};
		\node at (-14,-1.5) (bcd) {};
		\node at (-14,-4.5) (ab) {};
		\node at (-14,-7.5) (ac) {};
		
		\draw[->] (abcde) -- (abce);
		\draw[->] (abcde) -- (bcde);
		\draw[->] (abcde) -- (abcd);
		\draw[->] (abce) -- (bce);
		\draw[->] (bcde) -- (bce);
		\draw[->] (bcde) -- (bcd);
		\draw[->] (abcd) -- (bcd);
		\draw[->] (abc) -- (ab);
		\draw[->] (abc) -- (ac);
	\end{scope}
	
	\end{tikzpicture}
	\caption{The induced subcomplexes of $\Delta$ from \Cref{ex:intro}. The arrows indicate non-zero linear coefficients.}\label{fig:exintro}
\end{figure}

\begin{example}\label{ex:intro}
	Let $\Delta$ be the simplical complex with vertex set $\set{a,b,c,d,e}$ and facets $\set{a,c,d}$, $\set{b,d,e}$, $\set{c,d,e}$ and $\set{b,c}$.
	Its Stanley-Reisner ideal is $I_\Delta = \<ab,ac,bcd,bce\>$.
	A minimal free resolution $\Fb$ is given by the following complex:
	{\tiny
		\newcommand{\B}[1]{\mathbf{#1}}
		\[
		0 \leftarrow S
		\xleftarrow{\bordermatrix{
				& ab & ac & bcd & bce \cr
				1 & ab & ac & bcd & bce
		}}
		S^4 \xleftarrow{\bordermatrix{
				& abc    & abcd  & abce  & bcde \cr
			 ab & \B{-c} & cd    & ce    &  0 \cr
			 ac & \B{ d} &  0    &     0 &  0 \cr
			bcd &  0     & \B{-a} & \B{0} & \B{ e} \cr
			bce &  0     & \B{0} & \B{-a} & \B{-d} 
		}}
		S^4 \xleftarrow{\bordermatrix{
				& abcde \cr
				abc  &  0 \cr
				abcd & \B{ e} \cr
				abce & \B{-d} \cr
				bcde & \B{ a}
		}} S \leftarrow 0
		\]
	}
	The linear entries are marked in boldface.
	We indicate the relevant induced subcomplexes of $\Delta$ in \Cref{fig:exintro}.
	There, the arrows indicate non-zero linear entries in the matrices of $\Fb$.
	They correspond to non-zero restriction maps in the zero- or one-dimensional cohomology.
\end{example}

As a special case of \Cref{thm:main}, we obtain a very simple and explicit description of the $1$-linear strand of $\Fb$ (this is the strand containing the \emph{quadratic} generators of $I_\Delta$).
In particular, we show that the maps in the $1$-linear strand can always be written using only $\pm1$ coefficients, see \Cref{prop:2linear}. 
This extends and simplifies the results of Horwitz \cite{Horwitz} and Chen \cite{Chen}, who constructed the minimal free resolution of $I_\Delta$ under the assumption that $I_\Delta$ is generated by quadrics and has a linear resolution.

This article is structured as follows.
In \Cref{sec:prelim} we set up notational conventions and recall various preliminaries.
In the subsequent \Cref{sec:SR} we prove our main result.
In the last section, we ask several open questions and pose a conjecture.

\subsection*{Acknowledgments}
The author thanks Vic Reiner and Srikanth Iyengar for inspiring discussions.

\section{Notation and preliminaries}\label{sec:prelim}
For $n \in \NN$ we write $[n] := \set{1,\dotsc,n}$.
To simplify the notation, we set $U\setminus u := U\setminus\set{u}$ and $U\cup u := U\cup\set{u}$ for $U \subseteq [n]$ and $u \in [n]$.

Throughout the paper let $\kk$ denote a fixed field and $S = \kk[x_1, \dotsc, x_n]$ be a polynomial ring over it.
Further, we write
\[ \mm := \< x_1, \dotsc, x_n\> \]
for the unique maximal graded ideal in $S$.
We only consider the fine $\ZZ^n$-grading on $S$.
Squarefree multidegrees are identified with subsets of $[n]$. 
In particular, for $U \subseteq [n]$, we write $S(-U)$ for the free cyclic $S$-module whose generator is in degree $U$.

\subsection{The linear part}\label{sec:ld}
Let $M$ be a finitely generated graded $S$-module. We consider its minimal free resolution
\[ \Fb: 0 \oot M \oot  \Fb[0] \stackrel{\del_1}{\oot} \dotsb \stackrel{\del_n}{\oot} \Fb[n]\oot 0. \]
There is a natural filtration on $\Fb$, which is given by
\[ \mathcal{F}^j(\Fb[i]) := \mm^{j-i}\Fb[i]. \]
The associated graded complex $\Fblin$ is called the \emph{linear part} of $\Fb$.
It was introduced in \cite{EFS}, but see also \cite[Chapter 5]{HSV}. 
Note that $\Fblin[i] \cong \Fb[i]$ as $S$-modules, but the differentials on the complexes are different.
Indeed, $\Fblin$ can be constructed alternatively by choosing a basis for $\Fb$, representing its differential in this basis by matrices, and deleting all non-linear entries, that is, entries in $\mm^2$.

\subsection{Simplicial chains and cochains}
Let $\Delta$ be a simplicial complex with vertex set $[n]$.
For the convenience of the reader, we recall the definitions of the chain and cochain complexes of $\Delta$.
For keeping track of the signs, we use the notation
\[ \alpha(A,B) := \#\set{(a,b) \in A \times B \with a > b} \]
for subsets $A,B \subseteq [n]$. We further set $\alpha(a,B) = \alpha(\set{a},B)$.
The \emph{(augmented oriented) chain complex} of $\Delta$ is the complex of $\kk$-vector spaces $\tC_\bullet(\Delta)$, 
where $\tC_d(\Delta)$ is the $\kk$-vector space spanned by the $d$-faces of $\Delta$, and the differential is given by 
\[ \partial(F) = \sum_{i \in F} (-1)^{\alpha(i, F)} F \setminus\sset{i}. \]
Here, we consider the empty set as the unique face of dimension $-1$.
Note that the definition of $\alpha(i,F)$ depends on the ordering of $[n]$.
The \emph{(augmented oriented) cochain complex} of $\Delta$ is the dual complex $\tC^\bullet(\Delta) := \hom_\kk(\tC_\bullet(\Delta), \kk)$. 
We write $F^* \in \tC^d(\Delta)$ for the basis element dual to a $d$-face $F \in \Delta$.
In this basis, the differential on $\tC^\bullet(\Delta)$ can be written as
\[ \partial(F^*) = \sum_{i \in [n]\setminus F} (-1)^{\alpha(i, F )} (F\cup\sset{i})^*. \]
Here, we adopt the convention that $(F\cup\sset{i})^* = 0$ if $F\cup\sset{i} \notin \Delta$.
The (reduced) \emph{simplicial cohomology} of $\Delta$ is $\coho{*}{\Delta} := \coho{*}{\Delta; \kk} := H^*(\tC^\bullet(\Delta))$.

For a subcomplex $\Gamma \subseteq \Delta$, there 
is
a restriction map $\tC^\bullet(\Delta) \to \tC^\bullet(\Gamma)$.
If $\omega \in \tC^\bullet(\Delta)$ is a cochain and $U \subseteq [n]$, then we write $\restr{U}{\omega}$ for the restriction of $\omega$ to $\Delta_U$.

\section{Proof of the main result}\label{sec:SR}
Let $\Delta$ be a simplicial complex with vertex set $[n]$.
Recall that the \emph{Stanley-Reisner ideal} of $\Delta$ is defined as $I_\Delta := \<\xb^U \with U \subseteq [n], U \notin \Delta \>$, where $\xb^U := \prod_{i \in U} x_i$.
Further, the \emph{Stanley-Reisner ring} is $\kk[\Delta] := S/I_\Delta$. 
Every squarefree monomial ideal arises as the Stanley-Reisner ideal of some simplicial complex, see \cite[Theorem 1.7]{MS}.

We are going to need an explicit version of Hochster's formula.
It is of course well-known, but we give the details for the convenience of the reader.
Let $V = \Span_\kk\set{e_1, \dotsc, e_n}$ be an $n$-dimensional $\kk$-vector space and let $\Lambda^\bullet V$ denote the exterior algebra over it.
For $F = \set{i_1, \dotsc, i_r} \subseteq [n]$ with $i_1 < \dotsb< i_r$, we set $\eb_F := e_{i_1} \wedge \dotsm \wedge e_{i_r}$.
Then $\kk[\Delta] \otimes_\kk \Lambda^\bullet V$ is the Koszul complex of $\kk[\Delta]$.
\begin{proposition}[\cite{Hoc77}]\label{prop:hochster}
	For each squarefree multidegree $U \subseteq [n]$, there is an isomorphism of complexes
	$(\kk[\Delta] \otimes_\kk \Lambda^\bullet V)_U \too \tC^{\#U -1 - \bullet}(\Delta_U)$,
	given by $\xb^ F \otimes \eb_{U\setminus F} \mapsto (-1)^{\alpha(F,U)} F^*$.
\end{proposition}
\begin{proof}
	It suffices to show that the following diagram commutes:
	\[
	\begin{tikzcd} 
	\xb^ F \otimes \eb_{U\setminus F} 
	\arrow{r}{} \arrow[]{d} &
	\displaystyle{\sum_{i \in U\setminus F}} (-1)^{\alpha(i, U\setminus F)} \xb^F x_i \otimes \eb_{U\setminus(F\cup\sset{i})}  
	\arrow[]{d}\\
	(-1)^{\alpha(F,U)} F^*
	\arrow{r}{} &
	(-1)^{\alpha(F,U)}\displaystyle{\sum_{i \in U\setminus F}} (-1)^{\alpha(i,  F)} ( F\cup\sset{i})^*
	\end{tikzcd}
	\]
	We only need to show that
	$\alpha(F,U)+\alpha(i, F)\equiv \alpha(i,U\setminus F)+\alpha( F\cup i,U)$ modulo $2$.
	This follows from the following computation:
	\[ \alpha(F\cup i,U) - \alpha(F,U)
	= \alpha(i, U) =
	\alpha(i, F) + \alpha(i,U\setminus F) \qedhere
	\]
\end{proof}

Now we turn to the proof or \Cref{thm:main}, which we restate for convenience.
\thmmain*

\begin{proof}
	\newcommand{\ELL}{\mathcal{L}}
	\newcommand{\ELLp}{\mathcal{L}'}
	We follow the arguments of the proof of \cite[Theorem 4.1]{Y00}.
	Following \cite{HSV} and \cite[p. 107--109]{EFS}, we consider the double complex $(\ELL_{\bullet,\bullet}, \partial, \partial')$, whose modules are given by $\ELL_{a,b} := \kk[\Delta] \otimes_\kk \Lambda^a V \otimes_\kk S_b$ and the differentials are
	\[\begin{aligned}
	\partial(s_1 \otimes \eb_F \otimes s_2) &:= 
	\sum_{i \in F} (-1)^{\alpha(i,F)} s_1 x_i \otimes \eb_{F\setminus\sset{i}} \otimes \phantom{x_i}s_2
	\\
	\partial'(s_1 \otimes \eb_F \otimes s_2) &:= 
	\sum_{i \in F} (-1)^{\alpha(i,F)} s_1 \phantom{x_i}\otimes \eb_{F\setminus\sset{i}} \otimes x_i s_2
	\end{aligned}\]
	\noindent It is not difficult to see that the homology of $(\ELL_{\bullet,\bullet}, \partial)$ is isomorphic to $\Tor^S_\bullet(\kk[\Delta],\kk) \otimes_\kk S$.
	By \cite[Theorem 5.1]{HSV}, the linear part of the minimal free resolution is induced by $\partial'$.
	
	Consider the sub-double complex $\ELLp_{a,b} := \bigoplus_{\sigma \in [n]} (\kk[\Delta] \otimes_\kk \Lambda^a V \otimes_\kk S_b)_\sigma$ of $\ELL_{\bullet,\bullet}$. 
	As $\Tor^S_\bullet(\kk[\Delta],\kk)$ is non-zero in squarefree degrees only \cite[Cor. 1.40]{MS}, both $\ELLp_{\bullet,\bullet}$ and $\ELL_{\bullet,\bullet}$ have the same homology with respect to $\partial$.
	
	By \Cref{prop:hochster}, $(\ELLp_{\bullet,\bullet}, \partial)$ is isomorphic to $\bigoplus_{U \subseteq [n]} \tC^{\#U-1-\bullet}(\Delta_U) \otimes_\kk S(-U)$,
	where $\partial'$ translates to the map 
	\[\begin{aligned} \tC^{j}(\Delta_U) \otimes_\kk S(-U) &\too \bigoplus_{u \in U} \tC^{j}(\Delta_{U\setminus\sset{u}}) \otimes_\kk S(-U\setminus\sset{u}) \\
	F^* \otimes s &\longmapsto \sum_{u \in U} (-1)^{\alpha(u, F)} \restr{U\setminus u}{F^*} \otimes x_u s 
	\end{aligned}\]
	Now the claim follows by taking homology with respect to $\partial$ and applying \cite[Theorem 5.1]{HSV}.
\end{proof}

A particularly simple case of \Cref{thm:main} is the following. See \Cref{conj:3strand} for a conjectural improvement of this result.
\begin{corollary}\label{prop:2linear}
	Let $I \subseteq S$ be a monomial ideal  and let $\Fb$ be its minimal free resolution.
	Then one can choose a basis of $\Fb$ such that the maps in its $2$-linear strand have only coefficients in $\set{-1,0,1}$.
\end{corollary}
\begin{proof}
	We may assume that $I$ is squarefree by replacing it with its polarization \cite[p. 44]{MS}. 
	So it is the Stanley-Reisner ideal of some simplicial complex $\Delta$.
	By \Cref{thm:main}, maps in the $2$-linear strand of its minimal free resolution are induced by the restriction maps $\coho{0}{\Delta_U} \to \coho{0}{\Delta_{U\setminus\sset{u}}}$ for each $u \in U$.
	
	For each subset $U \subseteq [n]$ we choose a distinguished connected component $C_{U,0}$ of $\Delta_U$.
	For each other connected component $C_{U,i}$ of it, let $e_{U,i}: U \to \kk$ the function which is $1$ on the vertices of $C_{U,i}$ and $0$ on the others. 
	It is clear that the set $\set{e_{U,i}: i > 0}$ forms a basis of $\coho{0}{\Delta_U}$.
	
	We claim that in this basis, the differential has coefficients $\pm1$.
	For $i > 0$ there are the following cases:
	\begin{enumerate}
		\item $C_{U,i} = C_{U\setminus u, j}$ for some $j > 0$,
		\item $C_{U,i} = C_{U\setminus u, 0}$,
		\item $C_{U,i}$ splits into several connected components $C_{U\setminus u, j_1}, \dotsc, C_{U\setminus u, j_r}$ of $\Delta_{U\setminus u}$ with $j_1,\dotsc, j_r > 0$,
		\item same as (3), with $j_1 = 0$,
		\item $C_{U,i}$ is the isolated vertex $u$.
	\end{enumerate}
	In each case, it is easy to see that $e_{U,i}$ is mapped to a linear combination of the  $e_{U\setminus u,j}$ with coefficients in $\set{-1,0,1}$.
\end{proof}

\section{Questions and open problems}

\subsection{Affine monoid algebras}
Recall that a (positive) affine monoid $Q \subseteq \NN^n$ is a finitely generated submonoid of $\NN^n$.
The monoid algebras $\kk[Q]$ of affine monoids form a well-studied class of algebras.
We refer the reader to \cite{MS} or \cite[Chapter 6]{BH} for more information on these rings.
Each positive affine monoid has a unique minimal generating set, which is called its \emph{Hilbert basis}.
It yields a set of generators for $\kk[Q]$ and thus a surjection $S \to \kk[Q]$ from a polynomial ring $S$.
Moreover, $\kk[Q]$ carries a natural $\NN^n$-multigrading.
There is a combinatorial interpretation of the multigraded Betti numbers of $\kk[Q]$, namely $\Tor_i^S(\kk[Q], \kk)_\ab \cong \ho{i}{\Delta_\ab}$ for a certain simplicial complex $\Delta_\ab$, see \cite[Theorem 9.2]{MS}.

\begin{question}
	Is there a topological interpretation of the linear part of the minimal free resolution of $\kk[Q]$ over $S$?
\end{question}
In this situation, a description along the lines of \Cref{thm:main} would require a map $\ho{i}{\Delta_\ab} \to \ho{i-1}{\Delta_{\ab-\bb}}$, where $\bb$ is an element of the Hilbert basis such that $\ab-\bb \in Q$. Here, $\Delta_{\ab-\bb}$ is a subcomplex of $\Delta_{\ab}$,  but in general it is neither a restriction nor a link.

\subsection{Approximations of resolutions}
Let $I_\Delta \subseteq S$ be the Stanley-Reisner ideal of some simplicial complex $\Delta$ and let $\Fb$ denote the minimal free resolution of $I_\Delta$.
Hochster's formula can be interpreted as giving a description of the complex $\Fb/\mm\Fb$ (with trivial differential).
Our \Cref{thm:main} extends this by (essentially) describing $\Fb/\mm^2\Fb$.
These results can be considered as successive approximations of $\Fb$, so the following question seems natural:
\begin{question}
	Is there a combinatorial or topological description of $\Fb/\mm^3\Fb$?
\end{question}
This seems to be substantially more difficult than describing $\Fb/\mm^2\Fb$.
One reason for this is the following.
Even though a minimal free resolution is unique up to isomorphism, if one wants to write it down explicitly one needs to choose an $S$-basis for $\Fb$.
This choice can be done in two steps. 
First choose a $\kk$-basis for $\Fb/\mm\Fb = \Tor_*(S/I_\Delta,\kk)$, and then choose a lifting of these elements to $\Fb$ (any such lifting works due to Nakayama's lemma).
Hochster's formula is a convenient tool for the first choice.
\Cref{thm:main} implies that the differential of $\Fb/\mm^2\Fb$ does not depend on the second choice, but this is no longer true for $\Fb/\mm^3\Fb$.

\subsection{Coefficients in resolutions}
Let $I \subseteq S$ be a monomial ideal containing no variables, and let $\Fb$ be its with minimal free resolution.
We saw in \Cref{prop:2linear} that the differential in the $2$-linear strand of $\Fb$ can be written using only coefficients $\pm 1$. 
On the other hand, in \cite[Section 5]{RW} Reiner and Welker gave an example where the differential on the $4$-linear strand cannot be written using only coefficients $\pm 1$.
We believe that their example is optimal in that sense, and hence offer the following conjecture.

\begin{conjecture}\label{conj:3strand}
	Let $I \subseteq S$ be a monomial ideal.
	Then it is possible to choose a basis for its minimal free resolution $\Fb$, such that the differential on the $3$-linear strand can be written using only coefficients $\pm 1$.
\end{conjecture}

Note that the first map in $\Fb$, $\del:\Fb[1] \to\Fb[0]$, can always be written using coefficients from $\set{-1,0,1}$. This is easily seen by considering the Taylor resolution.
Further, it is not difficult to explicitly give a basis for $\Fb[2]$ such that the differential $\del:\Fb[2] \to\Fb[1]$ has coefficients $\pm1$.

\printbibliography
\end{document}